\documentclass{amsart}

\usepackage{amsmath,amssymb,amsthm,amsfonts, url}
\usepackage{setspace}

\newtheorem*{theorem*}{Theorem}
\newtheorem{theorem}{Theorem}

\newtheorem{prop}{Proposition}[section]
\newtheorem{lemma}[prop]{Lemma}
\newtheorem{coroll}[prop]{Corollary}

\newtheorem{remark}[prop]{Remark}

  \def\F{\mathcal{F}} 
 \def\I{\mathcal{I}}  
 \def\S{\mathcal{S}} 
\def\g{\mathfrak{g}}

\def\TT{\mathbb{T}}

 \def\LL{\mathcal{L}}  
\def\RR{\mathfrak{R}} 
\def\r{\mathfrak{r}}

\def\E{\mathbb{E}} \def\N{\mathbb{N}} \def\P{\mathbb{P}}
 \def\R{\mathbb{R}} \def\Z{\mathbb{Z}}

\def\Qsp{\mathcal{Q}} \def\Nsp{\mathcal{P}}

\def\n{\mathfrak{n}}
\def\X{\mathfrak{X}}
\def\XX{\hat{\X}}

\def\un{\mathbf{1}}

\def\Amb{\mathrm{Amb}}

\renewcommand{\le}{\leqslant}

\begin{document}

\author{Jean B\'{e}rard, Didier Piau}

\address[Jean B\'{e}rard]{\noindent Universit\'e de Lyon ;
Universit\'e Lyon 1 ;
Institut Camille Jordan CNRS UMR 5208 ;
43, boulevard du 11 novembre 1918,
F-69622 Villeurbanne Cedex; France.
 \newline e-mail: \rm
  \texttt{Jean.Berard@univ-lyon1.fr}}

\address[Didier Piau]{\noindent Institut Fourier - UMR 5582,
  Universit\'e Joseph Fourier Grenoble 1, 100 rue des Maths, BP 74,
  38402 Saint Martin d'H\`eres, France.
  \newline e-mail: \rm \texttt{Didier.Piau@ujf-grenoble.fr}}

\date{\today}

\title[CFTP times with ambiguities for particle systems]{Coupling from the past times 
with ambiguities and perturbations of interacting particle systems}



\begin{abstract}
We discuss coupling from the past techniques (CFTP) for perturbations of interacting particle systems on $\Z^d$ with a finite set of states, within the framework of the graphical construction of the dynamics based on Poisson processes. We first develop general results for what we call CFTP times with ambiguities. These are analogous to classical coupling (from the past) times,  except that the coupling property holds only provided that some ambiguities concerning the stochastic evolution of the system are resolved. If these ambiguities are rare enough on average, CFTP times with ambiguities can be used to build actual CFTP times, whose properties can be controlled in terms of those of the original CFTP time with ambiguities. We then prove a general perturbation result, which can be stated informally as follows. Start with an interacting particle system possessing a CFTP time  whose definition involves the exploration of an exponentially integrable number of points in the graphical construction, and which satisfies the positive rates property. Then consider a perturbation obtained by adding new transitions to the original dynamics.  Our result states that, provided that the perturbation is small enough (in the sense of small enough rates), the perturbed interacting particle system too possesses a CFTP time (with nice properties such as an exponentially decaying tail). 
The proof consists in defining a CFTP time with ambiguities for the perturbed dynamics, from the CFTP time for the unperturbed dynamics.  Finally, we discuss examples of particle systems to which this result can be applied. Concrete examples include a class of neighbor-dependent nucleotide substitution model, and variations of the classical voter model, illustrating the ability of our approach to go beyond the case of weakly interacting particle systems. 
\end{abstract}

\maketitle
%%%%%%%%%%%%%%%%%%%%%%%%%%%%%%%%%%%%%%%%%%%%%%%%
%
%%%%%%%%%%%%%%%%%%%%%%%%%%%%%%%%%%%%%%%%%%%%%%%%

\section{Introduction}

The present paper discusses coupling from the past (CFTP) techniques for interacting particle systems. The key idea of CFTP, as described in the seminal paper \cite{ProWil} by Propp and Wilson, consists in simulating coupled trajectories of a finite state-space Markov chain from further and further in the past, until eventually the present state of the Markov chain is the same for all trajectories, regardless of their starting point. One thus obtains an exact realization of the stationary distribution of the corresponding Markov chain, and, under a certain  monotonicity condition on the transitions of the chain, CFTP leads to a practical algorithm for sampling from the stationary distribution. Many extensions of this scheme have been developed since, notably to include processes on more general state-spaces, and situations where the monotonicity condition is not met (see the online bibliography \cite{Wil}).

Here, we consider interacting particle systems in the sense of \cite{Lig}, that is, continuous-time Markov processes describing the evolution of a system of states attached to the sites of $\Z^d$, the evolution at a site being governed by local transition rates involving the states of the neighboring sites. Our discussion is limited to particle systems with a finite state space whose dynamics can be prescribed by a finite family or transition rules (see below for a precise definition). For an ergodic particle system, one is interested in using CFTP to sample from the stationary distribution of the system. In general, it is not feasible to sample from the full stationary distribution, if only because a full configuration of the particle system is an infinite-dimensional object, comprising one definite state for each site of $\Z^d$. A more reasonable goal is the following: given a finite set of sites in $\Z^d$, use CFTP to sample from the marginal of the stationary distribution on this set of sites. This turns out to be possible when the interacting particle system possesses what we call in this paper a CFTP time, a precise definition being given below.

In the rest of this introduction, we give a quick overview of our results, some formal definitions and statements being postponed to later sections. A discussion of the existing literature and how the present work fits into it, is given at the end of the introduction.

\subsection{Interacting particle system dynamics and the graphical construction}

We consider interacting particle systems with a finite state space $\S$, whose set of sites is $\Z^d$ for some $d \geq 1$. To specify the dynamics of the system, we use the notion of a transition rule. Such a rule is a triple $\RR = (f, A, r)$, where $A$ is a finite subset of $\Z^d$, $f \ : \ \S^A \to S$ is a map, and $r \geq 0$ is a non-negative real number. Given a configuration of the system $\eta = (\eta(z))_{z \in \Z^d} \in \S^{\Z^d}$, 
and $x \in \Z^d$, we denote by $\RR^x \eta$ the configuration defined by $(\RR^x \eta)(x) = f(  (  \eta(x+y))_{y \in A}  )$ and  $(\RR^x \eta)(z) = \eta(z)$ for $z \neq x$. (Our convention when $A=\emptyset$ is that the set $\S^A$ is a singleton on which $f$ takes a single well-defined value.)  

Now, given a finite list of such transition rules $(\RR_i)_{i \in \I}=(f_i, A_i, r_i)_{i \in \I}$, we consider the interacting particle system dynamics specified by the infinitesimal generator $\LL$ satisfying 
\begin{equation}\label{e:inf-generator}\LL \phi (\eta) = \sum_{i \in \I} \sum_{x \in \Z^d} r_i \left(  \phi( \RR^x \eta  )  - \phi(\eta) \right),\end{equation} 
for all functions $\phi      \        :          \       \S^{\Z^d} \to \R$ that depend only  on a finite number of coordinates. Informally, \eqref{e:inf-generator} means that, during an infinitesimal time-interval of length $dt$, independently at each site $x \in \Z^d$, the transformation $\RR_i^x$ is applied to the current system configuration with probability $r_i dt$. 
It is standard to check (see \cite{Lig}) that \eqref{e:inf-generator} uniquely characterizes a c\`adl\`ag continuous-time Markov process $(\eta_t)_t$ on $\S^{\Z^d}$ equipped with the product topology.

In the sequel, we assume that the dynamics is in fact built through the graphical construction associated with the list of rules $(\RR_i)_{i \in \I}$ (see \cite{Lig2} for examples of this construction). Specifically, we consider a Poisson point process $\Nsp$ on $\Z^d  \times \I  \times \R$ with intensity $J$ defined by 
$dJ(x,i,t) := r_i d (c_{\Z^d} \otimes c_{\I}  \otimes \ell_{\R})(x,i,t)$, where $c_{\Z^d}$ and $c_{\I}$ denote the counting measure respectively on $\Z^d$ and $\I$, while $\ell_{R}$ denotes the Lebesgue measure on $\R$.
The realization of the point process $\Nsp$ prescribes the dynamics of the particle system through the fact that, for every $x$, $(\eta_t(x))_t$ is a jump process whose state may change only at times $t$ for which there exists an (a.s. unique) $i$ such that $(x,i,t) \in \Nsp$, and that, for such a $t$, one has    
\begin{equation}\label{e:graph-construct}  \eta_t = \RR_i^x ( \eta_{t-}) .\end{equation} 
Given $t_1 \leq t_2$ and $\xi \in \S^{\Z^d}$, define $\Phi_{t_1}^{t_2}(\xi)$ to be the configuration of the system at time $t_2$ obtained by starting in configuration  $\xi$ at time $t_1-$, and using the transitions specified by $\Nsp$ through \eqref{e:graph-construct}. We refer to  $\Phi$ as the stochastic flow on $\S^{\Z^d}$ induced by $\Nsp$. Given $t \leq 0$, we use the notation $\Nsp_t = \Nsp \cap (\Z^d \times [t,0[ \times \I )$, and let $\F_t = \sigma(\Nsp_t)$.

\subsection{CFTP times with and without ambiguities}

We now consider the coupling properties of the flow $\Phi$. We say that a negative and a.s. finite random variable $T$ is a CFTP time (for site zero) if the following property holds on the event\footnote{Throughout the paper, we adopt the convention that all possible pathologies of CFTP times and their variants are concentrated on the event that the corresponding time takes the value $-\infty$, so that one does not have to bother excluding undesirable zero probability events when the corresponding time takes a finite value.} $\{  T > -\infty   \}$,
\begin{equation}\label{e:coupling} \mbox{for all }\xi_1, \xi_2 \in \S^{\Z^d}, \  [ \Phi_{T}^{0-}(\xi_1) ](0)  =    [ \Phi_{T}^{0-}(\xi_2) ](0).\end{equation}

One checks that the existence of a CFTP time implies ergodicity of the particle system. Moreover, starting from an arbitrary configuration $\xi \in \S^{\Z^d}$, 
the distribution of $[\Phi_T^{0-}(\xi)](0)$ is exactly the marginal at site $0$ of  the corresponding stationary distribution. To obtain a sample from the marginal of this distribution on an arbitrary finite set of sites, one then has to repeat (after suitable translation) the procedure leading to $[\Phi_T^{0-}(\xi)](0)$ to each site in the set of interest.

The notion of CFTP time with ambiguities is a weakening of the notion of CFTP time, in which property \eqref{e:coupling} holds only when the "ambiguities" associated with the rules attached to a specific random subset $H$ of $\Nsp_T$, are resolved.   To give a precise definition, let us consider, for each $\alpha=(x,i,t) \in \Nsp$, and $s<t$, the random variable $e(\alpha,\xi,s)$ denoting the value at site $x$ produced by the application of the rule attached to $\alpha$ when starting in state $\xi$ at time $s-$, more formally:
\begin{equation} e(\alpha,s,\xi)  = [\Phi_s^t(\xi)](x).\end{equation}
When there exist two distinct $\xi_1, \xi_2$ such that $e(\alpha, s, \xi_1) \neq e(\alpha, s, \xi_2)$, we say that there is an ambiguity as to the result of the application of the rule attached to $\alpha$, when we start at time $s-$.
A coupling time with ambiguities consists of a negative a.s. finite random variable $T$, together with a random subset $H$ of $\Nsp_T$, which is assumed to be finite on the event $\{  T > -\infty   \}$, and such that the following  modification of \eqref{e:coupling} holds:
\begin{gather}\label{e:coupl-ambig} \mbox{ for all }\xi_1, \xi_2 \in \S^{\Z^d},  \ [ \Phi_{T}^{0-}(\xi_1) ](0)  =    [ \Phi_{T}^{0-}(\xi_2) ](0) \mbox{ provided that } \\ 
 e(\alpha, T, \xi_1) =    e(\alpha, T, \xi_2) \mbox{ for all }\alpha \in H. \nonumber \end{gather}
Note that, when $H= \emptyset$,  \eqref{e:coupl-ambig} is exactly \eqref{e:coupling}. In addition, 
we require that $H$ has the stopping property,  i.e. $H \cap \Nsp_t$ is $\F_t$-measurable for all $t$.

\subsection{Description of the main results}

Our first main result is that,  starting from a CFTP time with ambiguities $(T,H)$, one can build an actual CFTP time $T^*$, provided that $H$ contains few enough points on average. To give a precise statement, introduce the quantity 
\begin{equation}\label{d:growth} \g := \E \left( \sum_{(x,t,i) \in H}     |A_i|  \right).  \end{equation}
\begin{theorem}\label{t:finitude}
If $\g<1$, one can construct a CFTP time $T^*$ for the interacting particle system. 
\end{theorem}
The construction of $T^*$ is explained in Section \ref{s:definitions}.  Here is an informal description. Starting with  $\Amb_0 := \{ (0,0)  \}$, we recursively define a sequence $(\Amb_n)_{n \geq 0}$ of random subsets of $\Z^d \times \R$ in the following way. First, we apply the coupling time with ambiguities $(T,H)$ at each space-time point in $\Amb_n$. This generates a set of elements of $\Nsp$, with respect to which ambiguities have to be resolved. Then $\Amb_{n+1}$ is defined as the set of space-time points upon which the resolution of these ambiguities directly depends, i.e. for $\alpha=(x,i,t)$, the set $\{ (x+y, t);   \  y \in A_i \}$. The overall set of points generated by this process is  $\Amb_{\infty}   := \bigcup_{n \geq 0} \Amb_n,$ and $T^*$ is defined as the lowest value of $T$ obtained when applying the coupling time with ambiguities $(T,H)$ to the space-time points in $\Amb_{\infty}$. The idea underlying this construction is that, if $\Amb_{\infty}$ is finite, one can resolve ambiguities in a step-by-step manner, starting from the points in $\Amb_{\infty}$ that are furthest in the past and thus associated with an empty set of ambiguities, down to the origin where we can determine the value of $[\Phi_{T^*}^{0-}(\xi)](0)$ (a precise formulation is given in Proposition \ref{p:base-ambig} in Section \ref{s:preuve-1}). 

Without giving precise statements (see Theorems \ref{t:temps-exp} and \ref{t:espace-exp} in Section \ref{s:complete-results}), let us mention that, in addition to Theorem \ref{t:finitude}, it is possible to obtain estimates on the tail of $T^*$ and on  the range of its space-dependence (in terms of bounds on exponential moments), from analogous properties for $T$ and $H$.

Our second main result deals with perturbations of interacting particle systems. To formalize this notion, consider  a particle system whose dynamics is defined by a list of rules
$(\RR_i)_{i \in \I^u}$. This corresponds to the original, unperturbed, particle system. Then consider the dynamics defined by a list of rules of the form $(\RR_i)_{i \in \I^u \cup \I^p}$, where $\I^p$ is disjoint from $\I^u$. This corresponds to the perturbed particle system. Our result gives general conditions under which the existence of a CFTP time $T^u$ for the unperturbed particle system leads to the existence of a CFTP time for the perturbed particle system, provided that the perturbation is small enough. Our first condition is that the unperturbed dynamics possesses the positive rates property,  which means that, for every $v \in \S$, there exists a rule with index in $\I^u$ whose application inconditionally leads to the value $v$. Our second condition requires that the definition of $T^u$ involves the exploration of an exponentially integrable\footnote{For a non-negative random variable $X$, we say that $X$ is exponentially integrable if there exists $\mu > 0$ such that $E(e^{\mu X})<+\infty$.} number of points in $\Nsp$, a notion whose precise formulation is given in Section \ref{s:definitions}, and involves what we call the exploration process associated with $T^u$. Finally, the smallness of the perturbation is measured through two parameters $\epsilon$ and $\kappa$, that admit explicit definitions in terms of $(\RR_i)_{i \in \I^u}$ and $(\RR_i)_{i \in \I^p}$ (see Section \ref{ss:pos-rates-prop}).

\begin{theorem}\label{t:theoreme-principal}
Assume that the unperturbed dynamics has the positive rates property, and possesses a CFTP time associated with an exploration process whose total number of points is exponentially integrable. Then, for any perturbation with small enough $\epsilon$ and $\kappa$, we can construct a CFTP time with ambiguities $(T,H)$ for the perturbed dynamics, satisfying condition \eqref{d:growth}, so that the corresponding $T^*$ is a CFTP time for the perturbed dynamics. 
\end{theorem}

The construction of $(T,H)$ is explained in Section \ref{s:definitions}. Note that, in addition to satisfying \eqref{d:growth}, $(T,H)$ also satisfies the assumptions of Theorems \ref{t:temps-exp} and \ref{t:espace-exp} for sufficiently small $\epsilon$ and $\kappa$, leading to exponential moment bounds on the tail of $T^*$ and on the range of its space-dependence. This extension of Theorem \ref{t:theoreme-principal} is stated as Theorem \ref{t:theoreme-principal-extension} in Section \ref{s:complete-results}.

We illustrate Theorem \ref{t:theoreme-principal} with applications to several kinds of interacting particle systems. A first class of examples is given by systems that satisfy what we call the finite factor property (see Section \ref{s:applications}).  Informally, this means that the state of a site at a certain time depends on the points in the graphical construction and on the initial condition only through a window of fixed size around $x$. The simplest example is provided by dynamics for which distinct sites do not interact, whose perturbations correspond to weakly interacting interacting particle systems. A more elaborate concrete example is a family of nucleotide substitution models called RN+YpR (see \cite{BerGouPia}), which allows for arbitrarily strong interactions between neighboring sites and yet satisfies the finite factor property. Another class of systems to which we apply Theorem \ref{t:theoreme-principal}  (and for which the finite factor property is not satisfied), is given by noisy voter models on $\Z^d$. Specifically, we consider the case of a classical linear voter model  with an arbitrary finite alphabet, and a variation we call the voter model with asymmetric polling, that uses the classical $\{  +, - \}$ alphabet (see Section \ref{s:applications}). Note that these examples too extend beyond the weakly interacting case. 

Although we do not enter into the details here, let us note that the existence of a CFTP algorithm is not only interesting for simulation purposes, but can also provide interesting theoretical results on the particle system. Indeed, the existence of a CFTP time automatically implies that the interacting particle system is ergodic, and estimates on the tail of the CFTP time such as those provided by Theorem \ref{t:temps-exp} immediately lead to bounds on the speed of convergence to the stationary distribution. Similarly, estimates on the range of the space-dependence such as those provided by Theorem \ref{t:espace-exp}  yield bounds on the decay of spatial correlations. As a consequence, our results can be readily used to derive interesting conclusions about the behavior of the perturbed particle systems to which Theorem \ref{t:theoreme-principal} applies.

Also, note that, in this paper, we do not explicitly address issues related to the practical implementation of CFTP. However,  from the definition of $T^*$ in terms of $T$ and $H$ it should be clear that, if $T$ and $H$ lend themselves to an actual algorithmic implementation, so is the case for $T^*$. Similarly, Theorem \ref{t:theoreme-principal} yields an actual CFTP algorithm for the perturbed particle system provided that $T^u$ and the associated exploration process are compatible with an actual algorithmic implementation.

Finally, let us point out that a key role in the proof of our results is played by first-moment arguments, that allow us to essentially bypass the quite intricate analysis of dependencies that would otherwise be required to study the combinatorial objects we have to deal with (e.g. the sequence  of sets $(\Amb_n)_{n \geq 0}$, or  the tree-indexed  exploration process $\XX$ used to define a coupling time with ambiguities from $T^u$). In fact, first-moment arguments allow us to largely ignore these dependencies and obtain results in very much the same way as for classical branching processes.

\subsection{Discussion}

For ergodic particle systems satisfying a monotonicity condition similar to that of \cite{ProWil}, CFTP is always possible, as shown by  van den Berg and Steif in \cite{vdBSte}. 
For systems lacking monotonicity, CFTP algorithms have been developed under "high-noise" or "weak interaction" type assumptions, meaning that the strength of the interaction between neighboring sites has to be sufficiently small. In other words, the particle system under consideration has to be a sufficiently small perturbation of a 
system in which distinct sites do not interact. One example is given by Haggstr\"om and Steif \cite{HagSte}, who use a bounding set approach to control the coalescence of trajectories (see also \cite{dSaPic} for some refinements). Another example is given by Galves, Garcia, L\"ocherbach \cite{GalGarLoc} (see also \cite{GalGarLocOrl, GalLocOrl}), whose approach is based on a branching construction\footnote{A very similar construction was already used in \cite{FerFerGar} to devise CFTP algorithms in a different framework. In fact, various constructions of this kind appear in the literature, though not explicitly in the context of CFTP, see e.g. \cite{Fer, DobKryToo}.} of which the one we use in the present paper can be seen as a generalization. One interest of the present paper is that it provides a general criterion under which small perturbations of an interacting particle system retain some of the CFTP properties of the original unperturbed system, allowing one to go beyond the weakly interacting case. Let us mention that some of our applications overlap with the recent paper \cite{MohNewRav}, where a specific kind of perturbation of noisy voter models is considered, and ergodicity is proved for sufficiently small perturbations.  
Finally, note that the present paper is a revised and extended version of an earlier manuscript \cite{BerPia}, where the results were limited to perturbations of RN+YpR nucleotide substitution models.

\subsection{Organization of the paper}

The rest of the paper is organized as follows. Section \ref{s:definitions} contains the definitions of the various notions and objects that were encountered in the introduction but not formally defined. Section \ref{s:complete-results} contains additional results that were not stated in the introduction.  Section \ref{s:applications} contains the examples of application of Theorem \ref{t:theoreme-principal}. Section \ref{s:preuve-1} describes the proofs of Theorems  \ref{t:finitude}, \ref{t:temps-exp} and \ref{t:espace-exp}. Section  \ref{s:preuve-1} describes the proof of Theorem \ref{t:theoreme-principal} (and its extension Theorem \ref{t:theoreme-principal-extension}).

\section{Some formal definitions}\label{s:definitions}

\subsection{Canonical probability space}

 We assume throughout the paper that $\Nsp$ is defined on a canonical  probability space
$(\Omega, \F, \P)$ that we now describe. First, $\Omega$ is the set of locally finite subsets $\omega$ of 
$\Z^d \times \I  \times \R $ satisfying the additional requirements that (i) no two points in $\omega$ share the same $\R-$coordinate, (ii) for every $(x,i) \in \Z^d \times \I$,  both sets   $ \omega \cap (\{ (x,i) \} \times \R_+)$
and   $ \omega \cap (\{ (x,i) \} \times \R_-)$ are infinite, (iii) for every $(x,i,t) \in \omega$, and any sequence $(y_n,j_n,s_n)_{n \geq 0}$ starting at $( y_0,j_0,s_0 ) = (x,i,t)$ and satisfying $y_{n+1} \in y_n+A_{j_n}$ and $s_{n+1} < s_n$ for all $n \geq 0$, one has $\lim_{n \to +\infty} s_n  = -\infty$. Then $\F$ is defined as the $\sigma-$algebra generated by all the maps of the form $\omega \mapsto | \omega \cap B |$, where $B$  is a Borel subset of $\Z^d \times \I \times \R$. Finally, we set $\Nsp(\omega):=\omega$, and $\P$ is uniquely defined on $(\Omega, \F)$ by the requirement that $\Nsp$ is a Poisson process with intensity $J$ (this definition makes sense since (i), (ii), (iii) are almost sure properties of  a Poisson process with intensity $J$).  Note that thanks to (i), (ii), (iii),  $\Phi_{t_1}^{t_2}(\xi)$ is well-defined for any $t_1<t_2$, $\xi \in \S^{\Z^d}$ and $\omega \in \Omega$.

Given $(x,t) \in \Z^d \times \R$, we define the space-time shift $\tau_{x,t}$ on $\Omega$ by $\tau_{x,t}(\omega) := \bigcup_{ (z,i,s) \in \omega  } \{  (z-x,i,s-t) \}.$ (This definition is possible since (i), (ii), (iii) all are shift-invariant properties.) 
Given $a \in \N \cup \{ \pm \infty  \}$, we define $\Nsp^{-a,a} := \Nsp  \cap ([-a, a]^d \times \I \times \R) $, and let $\F^{-a,a} = \sigma(\Nsp^{-a,a})$. On occasions, we use the notation $\Nsp^{-a,a}_t := \Nsp \cap ([-a,a]^d \times [t,0[ \times \I )$.

Finally, to properly define the notion of a random subset of $\Nsp$, we introduce the  space $\tilde{\Omega}$ formed by the subsets of elements of $\Omega$, equipped with the  $\sigma-$algebra $\tilde{\F}$ defined just as $\F$. Note that, as a rule, in the rest of the paper, we mention measurability issues only when they involve a non-trivial argument.

In the sequel, we have to consider two distinct probability spaces, one associated with the unperturbed dynamics, and one with the perturbed dynamics. We denote by  $(\Omega^u, \F^u, \P^u)$ the probability space associated with the unperturbed dynamics specified by  $(\RR_i)_{i \in \I^u}$, while $(\Omega, \F, \P)$ refers to the perturbed dynamics specified by the full list  $(\RR_i)_{i \in \I}$. The corresponding Poisson processes are denoted respectively $\Nsp^u$ and $\Nsp$.

\subsection{Construction of $T^*$}

Given a coupling time with ambiguities $(T,H)$, 
define by induction the following  random sequence of subsets of $\Z^d \times ]-\infty, 0]$:
\begin{equation}\label{e:rec-amb} \Amb_0 := \{ (0,0)  \}, \    \Amb_{n+1} := \bigcup_{(x,t) \in \Amb_n}  \bigcup_{(z,i,s) \in \tau_{x,t}^{-1} \circ H \circ \tau_{x,t}} \bigcup_{y \in A_i} \{ (z+y,s)  \}.\end{equation}
Then let $\Amb_{\infty}   := \bigcup_{n \geq 0} \Amb_n,$
 and $T^*   :=  \inf_{(x,t) \in \Amb_{\infty}} t+T \circ \tau_{x,t}$ in the case where $|\Amb_{\infty}| < +\infty$, while $T^* := -\infty$  otherwise. 

\subsection{Exploration process}

Here, we define the notion of an exploration process attached to an interacting particle system.

Given a non-empty finite subset $B \subset \Z^d$ and $t \leq 0$, define $\n(B,t)$ to be the element of $\Nsp \cap (B \times \I \times ] - \infty , t [)$ with the highest time coordinate\footnote{Note that we are dealing with negative numbers here, so that the highest time coordinate here corresponds to the time coordinate with the least absolute value.}  (this is always well-defined with our choice of $\Omega$). Let $\Omega_f$ denote the set of all finite subsets of sets in $\Omega$, and let $\theta$ denote a measurable map from $\Omega_f$ to the set of finite subsets of  $\Z^d$ (where $\Omega_f$ is equipped with a $\sigma-$algebra defined as $\F$).

The exploration process associated to $\theta$ is defined as follows. Start with $\X_0:=\emptyset$, $\gamma_0:=0$. Then, for all $n \geq 0$, let $B_{n} := \theta(\X_n)$. If $B_n \neq \emptyset$, denote $\n(B_n,\gamma_n)=:(x_n, i_n, t_n)$, and let $\X_{n+1}:=  \X_n \cup \{ (x_n, i_n, t_n) \}$ and $\gamma_{n+1}:=t_n$. 
If $B_n = \emptyset$, then $\X_{n+1} := \X_n, \gamma_{n+1} := \gamma_n$. 
Let $\X_{\infty} := \bigcup_{n \geq 0} \X_n$ and $\gamma_{\infty} := \lim_{n \to +\infty} \gamma_n$.
The total number of points in the exploration process is then defined as $|\X_{\infty}|$.

We say that a CFTP time $T$ defined on $(\Omega, \F, \P)$ is associated with such an exploration process if, on $\{    |\X_{\infty}|<+\infty   \}$, one has $T = \gamma_{\infty}$, or equivalently, $T = \inf \{ t ;    \   (x,i,t) \in \X_{\infty} \}$, while $T=-\infty$ when $|\X_{\infty}|=+\infty $, and if, on $\{ T > -\infty \}$,  the value of $[\Phi^{0-}_{T}(\xi)](0)$, which is the same for every $\xi \in \S^{\Z^d}$, is measurable with respect to $\X_{\infty}$. 

We shall always assume that there exists a deterministic function $\beta \ : \ \N \to \N$ such that for every $\ell \geq 0$, 
\begin{equation}\label{e:borne-taille}  \theta(\X_{\ell}) \subset \{ -\beta (\ell), \ldots, \beta (\ell)   \}^d ,\end{equation}
and such that $\beta(\ell) = O(\ell)$ as $\ell$ goes to infinity.

In the sequel, we assume that a CFTP time $T^u$ for the unperturbed dynamics is defined on $(\Omega^u, \F^u, \P^u)$, and that $T^u$ is associated with an exploration process of the type we have just described. 
We denote the corresponding process by $\X^u$ to emphasize the fact that this process is defined for the unperturbed dynamics, but,  for the sake of readability, we use $\theta$, $x_n$, $i_n$, $t_n$, $\gamma_n,$ etc. instead of the awkward $\theta^u$, $x_n^u$, $i_n^u$, $t_n^u$, $\gamma_n^u$, etc. 

\begin{remark}
Given a CFTP time $T$ and an exploration process $(\X_n)_{n}$, the fact that  $T = \gamma_{\infty}$ does not in general imply that  $[\Phi^{0-}_{T}(\xi)](0)$ is measurable with respect  to $\X_{\infty}$, so this last condition has to be added to the definition of an exploration process associated to a CFTP time.
\end{remark}

\subsection{Positive rates property, $\epsilon$ and $\kappa$}\label{ss:pos-rates-prop}

The positive rates property for  the set of non-perturbative rules $(\RR_i)_{i \in \I^u}$ means that, for every $v \in \S$, there exists a rule with index in $\I^u$ which is of the form $(A,f,r)$ with $r>0$, $A=\emptyset$ and $f \equiv v$. We denote by $\iota_v$ the\footnote{One may assume without loss of generality that, for any $v$,  there is a single such rule, since identical rules with distinct indices may always be merged into a single rule.} index of such a rule. We then control the smallness of the perturbation of  $(\RR_i)_{i \in \I^u}$ by $(\RR_i)_{i \in \I^p}$ through the following two parameters 
\begin{equation}\label{e:def-epsilon}\epsilon := \sup_{v \in \S}    \left( \sum_{j \in \I^p  ;    \   v \in f_j(A_j)}   r_j    \right)    (r_{\iota_v})^{-1},\end{equation}
\begin{equation}\label{e:def-kappa}  \kappa:=  \left( \sum_{i \in \I^p}  | A_i | r_i  \right)   \left(  \sum_{i \in \I} r_i    \right)^{-1}.  \end{equation}

\subsection{Construction of a coupling time with ambiguities $(T,H)$ from $T^u$}

We now define what we call the exploration process with locking of perturbative ambiguities attached to the perturbed dynamics, and associated to the map $\theta$  used to define the exploration process $\X^u$ of the unperturbed dynamics. This is the process we use to define a coupling time with ambiguities for the perturbed dynamics. 

Informally, the construction can be described as follows: run the exploration process associated with the unperturbed dynamics on $\Nsp$ (instead of $\Nsp^u$).   When an $\alpha=(x,i,t)$ corresponding to a perturbative rule, i.e. $i \in \I^p$ is encountered, split the exploration process into $|f(A_i)|$ exploration processes evolving in parallel, one for each $v \in f(A_i)$, in which $(x,i,t)$ is replaced by $(x, \iota_v, t)$. 

The formal construction uses  a recursively constructed tree $\TT$ to label the process. Let us start with the root of $\TT$, denoted $r$, for which we set $\XX_r := \emptyset $ and $\gamma_r:=0$. Then, for any vertex $a \in \TT$, we set  $B_a := \theta(\XX_a)$. Then, if $B_a \neq \emptyset$, 
denote $\n(B_a,\gamma_a)=(x_a, i_a, t_a)$. If $i_a \in \I^u$, we attach a single son $b$ to $a$, and 
let $\XX_{b}:=  \XX_a \cup \{ (x_a, i_a, t_a) \}$ and $\gamma_{b}:=t_a$. On the other hand, if $i_a \in \I^p$, we attach to $a$ a list of sons $\left(b_v,    \  v \in f_{i_a}(\S^{A_{i_a}}) \right)$, and let  
 $\XX_{b_v}:=  \XX_a \cup \{ (x_a, \iota_v,t_a) \}$ and $\gamma_{b_v}:=t_a$. If $B_a = \emptyset$, then $a$ has no son. Finally, we let $\XX_{\infty} := \bigcup_{a \in \TT} \XX_a$.

We now define  $T$ by \begin{equation}\label{d:def-T}T := \inf \{ t ;    \    (x,i,t) \in    \XX_{\infty}    \} \mbox{ if } |  \XX_{\infty}  | < +\infty,\end{equation}
while $T:=-\infty$ if  $|  \XX_{\infty}  | = +\infty$, 
and $H$ is defined by 
\begin{equation}\label{d:def-H}H := \{  (x_a, i_a, t_a);    \       a \in \TT', \  i_a \in \I^p   \},\end{equation}
where $\TT'$ denotes the subtree of $\TT$ obtained by removing the leaves of $\TT$. 
Note that one may view $\XX_{\infty}$ as the result of the exploration process associated to a certain map $\hat{\theta}$. However, the representation with a labelled tree turns out to be more convenient  for our purposes.

\begin{remark}
The definition of $\XX_{\infty}$ makes sense whether or not $| \X^u_{\infty} |$ has an exponentially decaying tail, as assumed in the statement of Theorem \ref{t:theoreme-principal}: provided that $|\XX_{\infty}|$ is a.s. finite, $(T,H)$ is indeed a CFTP time with ambiguities (see Proposition \ref{p:base-perturb}), and the role of the tail condition on $| \X^u_{\infty} |$ in Theorem \ref{t:theoreme-principal} is indeed to ensure that $|\XX_{\infty}|$ is a.s. finite.
\end{remark}

\section{Additional results}\label{s:complete-results}

The first result shows that the exponential moments of $T^*$ can be controlled in terms of similar moments for $T$ and $H$.
 
For $\lambda \in  \R$, define
\begin{equation}\label{e:def-Laplace-temps}\Lambda_T(\lambda) := \E(\exp(\lambda T)), \ \Lambda_{H,time}(\lambda) := \E \left( \sum_{(x,i,t) \in H} |A_i| \exp(\lambda t)    \right).\end{equation}
\begin{theorem}\label{t:temps-exp}
Assume that $\lambda \le 0$ is such that   $ \Lambda_T(\lambda) < +\infty$ and  $\Lambda_{H,time}(\lambda) < 1$. Then
$$\E(\exp(\lambda T^*)) \leq \Lambda_T(\lambda) (1 - \Lambda_{H,time}(\lambda))^{-1}.$$ 
\end{theorem}

Our next result deals with the range of space-dependence of $T^*$. To formalize this notion,  say that a $\N \cup \{ + \infty  \}$-valued random variable $L$ defines a stopping box in $\Z^d$ if, for any $a \in \N$, one has $\{L = a \} \in \F^{-a,a}$. We say that an a.s. finite such random variable bounds the width of a CFTP time $T$ if, on $\{ T > -\infty \}$,  the  value of $[\Phi_{T}^{0-}(\xi)](0)$ (which by definition 
does not depend on the choice of $\xi \in \S^{\Z^d}$) is measurable with respect to $\F^{-L,L}$. We say that $L$ 
bounds the width of a CFTP time with ambiguities if $H$ is measurable with respect to $\F^{-L, L}$
and if there exists a measurable map\footnote{To be more specific about measurability assumptions concerning $(e(\alpha, T, \xi))_{\alpha \in H}$, we assume that it is encoded as the random subset of $\Z^d \times \I \times \R \times \S$ defined by $\bigcup _{\alpha \in H}   (\alpha, e(\alpha, T, \xi))$, where the $\sigma-$algebra on the set of locally finite subsets of  $\Z^d \times \I \times \R \times \S$ is generated by maps of the form $\varpi \to  | \varpi \cap  B|$, where $B$ is a  Borel subset of $\Z^d \times \I \times \R \times \S$.} $\Theta$ such that, on $\{  T > -\infty \}$, for all $\xi \in \S^{\Z^d}$, $[\Phi_{T}^{0-}(\xi)](0) = \Theta\left(L,  \Nsp^{-L, L} ,  (e(\alpha, T, \xi))_{\alpha \in H} \right) $.

For $\lambda \in \R$, and $1 \leq q \leq d$, define
\begin{equation}\label{e:def-Laplace-esp}\Lambda_{L}(\lambda) := \E(e^{\lambda L}), \ \Lambda_{H,space}(\lambda,q) := \E \left( \sum_{(x,t,i) \in H} \sum_{z \in A_i}  e^{\lambda (x_q+z_q)}    \right).\end{equation}
Then define $L^*_+$ by $L^*_+ := \sup_{1 \leq q \leq d} \sup_{ (x,t) \in \Amb_{\infty} } x_q + L \circ \tau_{x,t}$, and $L^*_-$ by $L^*_- := \inf_{1 \leq q \leq d} \inf_{ (x,t) \in \Amb_{\infty} } x_q - L \circ \tau_{x,t}$. Finally, let $L^* := \max(L^+, -L^-)$.
\begin{theorem}\label{t:espace-exp}
If $\g<1$, and $L$ bounds the width of $(T,H)$ then $L^*$ bounds the width of $T^*$. Moreover, if $\lambda > 0$ is such that $ \Lambda_{L}(\lambda)  < +\infty$ and $\Lambda_{H,space}( \pm \lambda,q) < 1$ for all $q$,  then 
$$\E(\exp(\lambda L_+^*)) \leq \Lambda_{L}(\lambda)  \sup_{1 \leq q \leq d} (1 - \Lambda_{H,space}(\lambda,q))^{-1},$$
$$\E(\exp(-\lambda L_-^*)) \leq \Lambda_{L}(\lambda)   \sup_{1 \leq q \leq d}  (1 - \Lambda_{H,space}(-\lambda,q))^{-1}.$$
\end{theorem}

Finally, we have the following extension of Theorem \ref{t:theoreme-principal}.

\begin{theorem}[Extension of Theorem \ref{t:theoreme-principal}]\label{t:theoreme-principal-extension}
Under the assumptions of Theorem \ref{t:theoreme-principal},  for any list of perturbative rules with small enough $\epsilon$ and $\kappa$, the pair $(T,H)$ defined by \eqref{d:def-T} and \eqref{d:def-H} defines a  CFTP time with ambiguities that satisfies the assumptions of Theorem \ref{t:finitude}, together with the assumptions of  Theorems \ref{t:temps-exp} and \ref{t:espace-exp} for small enough $|\lambda|$.
\end{theorem}

For the sake of readability, we did not include explicit estimates in the statement of Theorem \ref{t:theoreme-principal-extension}. However, looking at the proofs given in Section \ref{s:preuve-2}, it is easy to obtain explicit control 
upon the characteristics of $(T,H)$ (namely, $\g$,  $\Lambda_T$, $\Lambda_{H, time}$, $\Lambda_L$, $\Lambda_{H, space}$) in terms of $\epsilon, \kappa$, and the parameters $(\RR_i)_{i \in \I^u}$ of the unperturbed model. Combined with Theorems \ref{t:finitude}, \ref{t:temps-exp} and \ref{t:espace-exp}, this leads to an explicit control upon the characteristics of the resulting CFTP time $T^*$.

\section{Applications}\label{s:applications}

In this section, we give some examples of dynamics which satisfy the properties required for the unperturbed dynamics in Theorem \ref{t:theoreme-principal}. Since we discuss unperturbed dynamics only, it is unnecessary to use $^u$ superscripts to distinguish between perturbed and unperturbed dynamics, and consequently such superscripts are not used in this section.

\subsection{Perturbations of finite factor models}

We say that the dynamics  possess the finite factor property if there exists $b \in \N$ such that, for all $t < 0$, 
$\left[\Phi^{0-}_{t}(\xi)\right] $  is measurable with respect to $\Nsp^{-b,b}_t$ and $\left(\xi(x);   \   x \in \{-b, \ldots, b\}^d \right)$. 

\begin{prop}\label{p:finite-factor} Any dynamics with the finite factor property and the positive rates property satisfies the assumptions of Theorem \ref{t:theoreme-principal}, i.e. there exists a CFTP time associated with an exploration process whose total size has some finite exponential moment.
\end{prop}

\begin{proof}
Given $X \in \Omega_f$, let $q:=|X|$, and denote by  $(y_k, j_k, s_k)_{0 \leq k \leq q-1}$ the list of elements of $X$, indexed by decreasing order of time, so that $s_0>\cdots > s_{q-1}$. Let also $h:= (2b+1)^d$. Now  consider the exploration process associated with the map $\theta$ defined as follows. Set  $\theta(X) := \emptyset$ when the following three conditions are met 
\begin{itemize}
\item[a)]  $|X| \geq h$,
\item[b)] $\{  y_{q-1},\ldots, y_{q-h}   \} = \{ -b, \ldots, b \}^d $,
\item[c)] for all $q-h \leq k \leq q-1$, $A_{j_k} = \emptyset$,
\end{itemize} 
Otherwise, let $\theta(X) := \{-b, \ldots, b\}^d$.
Denote by $(\X_n)_n$ the corresponding exploration process, and observe that condition \eqref{e:borne-taille} is satisfied with $\beta(\ell) := b$ for all $\ell$.
One checks that given $\X_n$, the probability that $\X_{n+h}$ satisfies $\theta(\X_{n+h}) = \emptyset$ is bounded below by $h ! h^{-h}  \rho^h$, 
where $\rho :=  (\sum_{i \in \I}   r_i \un(A_i = \emptyset)  )(\sum_{i \in \I} r_i)^{-1}$.
This proves the fact that there exists $\mu>0$ such that 
 $\E\left( e^{\mu | \X_{\infty}} | \right) < +\infty$.  
Let us now check that $T:=\gamma_{\infty}$ is  a CFTP time for the dynamics, associated with the exploration process defined by $\theta$.
 Define $U$ to be the a.s. finite smallest index $k$ such that $\theta(\X_k) = \emptyset$. From conditions a) b) c), one has that,  on $\{ U < +\infty \}$, for all $x \in  \{ -b, \ldots, b \}^d$, 
 $[\Phi_{T}^{\gamma_{U-h}-}(\xi)](x)$ takes the same value for every $\xi$, and this value is measurable with respect to $(y_{U}, j_{U}), \ldots, (y_{U-h+1}, j_{U-h+1})$. On the other hand, the fact that $\theta(\X_k) =    \{ -b, \ldots, b \}^d$ for all $k \leq U-h$ shows that $\X_{U-h} = \Nsp^{-b,b}_{\gamma_{U-h}} $. Now by the definition of the flow, one has that
 $$[\Phi_T^{0-}(\xi)](0) =  \left[\Phi^{0-}_{\gamma_{U-h}} \left(\Phi_T^{\gamma_{U-h}-}(\xi)\right) \right](0).$$ The finite factor property then yields that $[\Phi_T^{0-}(\xi)](0)$ is the same whatever the value of $\xi$, and that this value is measurable with respect to $\X_{U}$.
\end{proof}

The simplest example of dynamics with the finite factor property is the case where distinct sites  do not interact, i.e. $A_i \subset \{ 0 \}$ for every $i \in \I$. In this case, the factor property holds with $b=0$, and each site evolves independently according to a continuous-time Markov chain on $\S$.

A more sophisticated example, whose study was our original motivation for this work, is the so-called class of RN+YpR nucleotide substitution models,  see \cite{BerGouPia}, whose goal is to provide tractable models that include neighbor-dependent effects such as the well-known hypermutability of CpG dinucleotides. These models  use  the nucleotidic alphabet $\S := \{  A, C, G, T   \}$ as their state space, and $\Z$ as their set of sites, with $\S^{\Z}$   
being an idealized representation of a DNA sequence. Additionally, $\S$ is divided into the set of pyrimidines $Y:= \{  C, T  \}$, and purines $R := \{  A, G \}$, and we say that $Y$ is the type of $C$ and $T$, while $R$ is the type of $A$ and $G$.  

The RN+YpR dynamics is specified through the following list of rules (each rule is of the form  $(f,A,r)$):
\begin{itemize}
\item unconditional rules: for each $v \in \S$, a rule with $A := \emptyset$, $r>0$ and $f \equiv v$;
\item transversion rules: for each $v \in \S$, a rule with $A: = \{ 0 \}$ and $f(w) := v$ if $v$ and $w$ are not of the same type, $f(w):=w$ otherwise; 
\item transition rules:  for each $v \in \S$, a rule with $A := \{ 0 \}$ and $f(w) := v$ if $v$ and $w$ are of the same type, $f(w):=w$ otherwise; 
\item left-dependent rules: for each $u \in Y$, $v \in R$, $v' \in R$, a rule with $A := \{ -1, 0 \}$, 
$f(w_{-1},w_0):=v'$ if  $(w_{-1},w_0)=(u,v)$,  $f(w_{-1},w_0):=w_0$ otherwise;    
\item right-dependent rules:  for each $u \in Y$, $v \in R$, $u' \in Y$, a rule with $A := \{  0, 1 \}$, 
$f(w_0,w_1):=u'$ if  $(w_0,w_1)=(u,v)$,  $f(w_0,w_1):=w_0$ otherwise.    
\end{itemize}

It turns out (see \cite{BerGouPia}) that RN+YpR models have the finite factor property with $b:=1$. Let us insist that the rates of left- and right-dependent rules, whence the strength of the interaction between sites, may be arbitrarily large, so that the RN+YpR class contains models that are not weakly dependent. Note that one can generalize this class of models to produce interacting particle systems with an arbitrarily long range of dependence, where the minimal $b$ for which the finite factor property holds can be made arbitrarily large, although these seem less biologically motivated. Note also that, in the case of the RN+YpR model, one can define alternative coupling times which, as opposed to the one defined in the proof of Proposition \ref{p:finite-factor}, do not get larger and larger when the interaction strength (given by the rates of the rules involving interactions between neighboring sites) gets large,  see \cite{BerPia}.

\subsection{Perturbations of voter-like models}

We now describe how Theorem \ref{t:theoreme-principal} can be applied to variants of classical interacting particle systems such as the voter model on $\Z^d$ (see \cite{Lig2}).

\subsubsection{Classical linear voter model.}

Let $p(\cdot)$ denote a probability measure on $\Z^d$ with finite support. The dynamics of the classical voter model can be defined thanks to the following set of rules:
\begin{itemize}
\item state-copying rules: for each $x$ in the support of $p(\cdot)$, a rule with $A := \{ 0,x \}$ and $r:=p(x)$, with $f(w_0, w_x):=w_x$.
\end{itemize}
One might interpret this model as describing the evolution of opinions of individuals attached to the sites of $\Z^d$, with $\S$ representing the set of possible opinions. The individual at $x$ waits for a unit exponential time, then chooses a random location $y \in \Z^d$ with probability $p(y-x)$, and adopts the opinion of the individual attached to site $y$. As such, the voter model does not satisfy the assumptions of Theorem \ref{t:theoreme-principal}, since it does not enjoy the positive rates property. As a consequence, we add to this model a list of unconditional rules so as to enforce this property:  
\begin{itemize}
\item unconditional rules: for each $v \in \S$, a rule with $A := \emptyset$, $r>0$ and $f \equiv v$.
\end{itemize}
We call the resulting model "noisy voter model". Note that this addition dramatically changes the dynamics of the voter model, since it automatically turns it into an ergodic interacting particle system. Note that we may consider this addition as part of the perturbation of the original voter model we want to study, but this part of the perturbation has to be included in the dynamics prior to the application of Theorem  \ref{t:theoreme-principal}.

\begin{prop}
The noisy voter model satisfies the assumptions of Theorem \ref{t:theoreme-principal}.
\end{prop}

\begin{proof}
The corresponding exploration process is defined as follows. First $\theta(\emptyset) := \{ 0 \}$. 
Then, given a non-empty $X \in \Omega_f$, denote by $(y,j,s)$ the element of $X$ with the lowest time-coordinate. Then let $\theta(X):=\emptyset$ if $A_j = \emptyset$. Otherwise, $A_j$ is of the form $\{ 0, x  \}$, and we let $\theta(X):= \{ y+x \}$. We denote by $(\X_n)_{n}$ the corresponding exploration process. Note that condition \eqref{e:borne-taille} is satisfied with $\beta(\ell) := \sup \{  |z|;    \ p(z) \neq 0   \} \times \ell$. To prove that the $|\X_{\infty}|$ has some finite exponential moment, note that, conditional upon $\X_{n}$, if $\theta(\X_n) \neq \emptyset$, the probability that the next point to be included in $\X_{n+1}$ corresponds to a rule of the form $f \equiv v$ for some $v$, is bounded below by the ratio  $(\sum_{v \in \S} r_{\iota_v}) (\sum_{i \in \I} r_i)^{-1}$. Since in this case $\theta(\X_{n+1})=\emptyset$, a geometric upper bound holds for the tail of $|\X_{\infty}|$. 
Let us now check that $T$ defined as the least time-coordinate of an element in $\X_{\infty}$ is  a CFTP time for the dynamics, associated with the exploration process defined by $\theta$. Indeed, it is clear from the definition of the dynamics that if  the element of  $X_{\infty}$ with the least time-coordinate is associated with the rule $\iota_v$, then $\left[\Phi_{T}^{0-}(\xi)\right](0)=v$ for all $\xi$.
\end{proof}

\subsubsection{Voter model with asymmetric polling} 

We now consider a variation upon the classical voter model. Let $A^{(1)},\ldots, A^{(m)}$ denote a list of finite non-empty subsets of $\Z^d$, $r^{(1)},\ldots, r^{(m)}$ denote a list of non-negative real numbers, and  take as a state space $S := \{  +, -  \}$.
The set of rules characterizing our model is the following: 
\begin{itemize}
\item polling rules:  for each $1 \leq i \leq m$, a rule with $A:=A^{(i)}$ and 
$f(w):=+$ if $w_x = +$ for at least one $x \in A^{(i)}$, $f(w):=-$ otherwise.
\end{itemize}
We call this model the voter model with asymmetric polling. Here, an individual performs a poll within a randomly chosen finite subset of individuals, and adopts an opinion that depends on the results of the poll in an asymmetric way: indeed, the individual will adopt the opinion denoted $+$ if any of the individuals in the poll expresses the opinion $+$, while, to adopt the opinion denoted $-$, consensus within the poll is required.

As in the case of the classical linear voter model, we add to the above set of rules a list of unconditional rules ensuring the positive rates property:
\begin{itemize}
\item unconditional rules: for $v=+,-$, a rule with $A := \emptyset$, $r>0$ and $f \equiv v$.
\end{itemize}
The resulting model is called noisy voter model with asymmetric polling.

\begin{prop}
The noisy voter model with asymmetric polling satisfies the assumptions of Theorem \ref{t:theoreme-principal}.
\end{prop}

\begin{proof}
The corresponding exploration process is defined as follows. First $\theta(\emptyset) := \{ 0 \}$. 
Then, given a non-empty set $X \in \Omega_f$, denote by $(y,j,s)$ the element of $X$ with the least time-coordinate. Then let $\theta(X):=\emptyset$ if $A_j = \emptyset$ and $f_j \equiv +$. 
If $A_j = \emptyset$ and $f_j \equiv -$, then let $\theta(X) := \theta(X \setminus \{   (y,j,s)   \}) \setminus \{ j \}$. Otherwise, $A_j$ is of the form $A^{(k)}$ for some $1 \leq k \leq m$, and we let  
 $\theta(X) := \theta(X \setminus \{   (y,j,s)   \}) \cup (y+A^{(k)})$.
 We denote by $(\X_n)_{n}$ the corresponding exploration process. Note that condition \eqref{e:borne-taille} is satisfied with $\beta(\ell) := \sup \{  |z|;      z \in \cup_{1 \leq k \leq m} A^{(k)} \} \times \ell$. 

To prove that the number of points in $|\X_{\infty}|$ has some finite exponential moment, note that, conditional upon $\X_{n}$, if $\theta(\X_n) \neq \emptyset$, the probability that 
the next point to be included in $\X_{n+1}$ corresponds to the rule with $f \equiv +$, is bounded below by the ratio  $r_{\iota_+} (\sum_{i \in \I} r_i)^{-1}$. Since in this case 
$\theta(X_{n+1})=\emptyset$, a geometric upper bound holds for the tail of $|\X_{\infty}|$. 
Let us now check that $T$ defined as the least time-coordinate of an element in $\X_{\infty}$ is  a CFTP time for the dynamics, associated with the exploration process defined by $\theta$. Indeed, it is clear from the definition of the dynamics that if 
the element of  $\X_{\infty}$ with the least time coordinate is associated with the rule $\iota_+$, $\left[\Phi_{T}^{0-}(\xi)\right](0)=+$ for all $\xi$, while, if this element is associated with 
 the rule $\iota_-$, $\left[\Phi_{T}^{0-}(\xi)\right](0)=-$ for all $\xi$.
\end{proof}

\section{Proofs of Theorems \ref{t:finitude}, \ref{t:temps-exp} and \ref{t:espace-exp}}\label{s:preuve-1}

We start with a proposition showing that, if $T^*$ is finite with probability one, then $T^*$ is indeed  a CFTP time. 

\begin{prop}\label{p:base-ambig}
If $\P(T^* > - \infty)=1$, then $T^*$ is a CFTP time.
\end{prop}

\begin{proof}
Note that, with our definitions, $T^*>-\infty$ implies that $|\Amb_{\infty}|<+\infty$.
The proof is by induction on $|\Amb_{\infty}|$. Specifically, we shall show for all $n \geq 0$ that the following property $(P_n)$ is true: on $\{ |\Amb_{\infty}|=n  ,   \  T^* > - \infty\}$, for all $\xi_1, \xi_2 \in \S^{\Z^d}$, 
$[\Phi_{T^*}^{0-}(\xi_1)](0) = [\Phi_{T^*}^{0-}(\xi_2)](0)$. 
Assume throughout that $T^* > -\infty$, and let us start with $n=1$. If  $|\Amb_{\infty}|=1$, a first possibility is that 
$H = \emptyset$. In this case, the definition of a CFTP time with ambiguities shows that $[\Phi_{T}^{0-}(\xi_1)](0) = [\Phi_{T}^{0-}(\xi_2)](0)$ for all $\xi_1, \xi_2 \in \S^{\Z^d}$, whence, since by definition $T^* \leq T$, the fact that $[\Phi_{T^*}^{0-}(\xi_1)](0) = [\Phi_{T^*}^{0-}(\xi_2)](0)$ for all $\xi_1, \xi_2 \in \S^{\Z^d}$.
If $H \neq \emptyset$, the fact that $|\Amb_{\infty}|=1$ shows that every $\alpha=(x,i,t) \in H$ is such that 
$A_i = \emptyset$. In this case, for any $s \leq t$, $e(\alpha,\xi,s)$ depends neither on $\xi$ nor $s$, so that again $[\Phi_{T^*}^{0-}(\xi_1)](0) = [\Phi_{T^*}^{0-}(\xi_2)](0)$ for all $\xi_1, \xi_2 \in \S^{\Z^d}$.
We now show that  $(P_{n+1})$ is valid provided that $(P_k)$ is valid for all $1 \leq k \leq n$. Assume that  $|\Amb_{\infty}|=n+1$. It is enough to prove that, for any  $\alpha = (z,i,t) \in H$ such that $A_i \neq \emptyset$, $e(\alpha, T^*, \xi)$ admits the same value for every $\xi \in \S^{\Z^d}$.  
Consider such an $\alpha=(z,i,t)$,   let $y \in A_i$, and $x:=z+y$. Then observe that,  by definition, 
\begin{equation}\label{e:inclusion-translation}\tau_{x,t}^{-1} \circ \Amb_{\infty} \circ \tau_{x,t} \subset  \bigcup_{k \geq 1} \Amb_k.\end{equation}
Since  $\Amb_{0} =  \{ (0,0)  \}$ while  $(0,0)  \notin \bigcup_{k \geq 1} \Amb_k$,  our assumption that $|\Amb_{\infty}|=n+1$ implies that $|  \bigcup_{k \geq 1} \Amb_k  | = n$.  We thus deduce from \eqref{e:inclusion-translation} that $|  \Amb_{\infty} \circ \tau_{x,t} | \leq n$. Moreover, \eqref{e:inclusion-translation} shows that $t+T^* \circ \tau_{x,t} \geq T^*$, so that our assumption that 
$T^*>-\infty$ implies that $T^* \circ \tau_{x,t} > -\infty$. Our induction hypothesis then implies that   for all $\xi_1, \xi_2 \in \S^{\Z^d}$, $[\Phi_{T^*}^{0-}(\xi_1)](0) \circ \tau_{x,t}= [\Phi_{T^*}^{0-}(\xi_2)](0) \circ \tau_{x,t}$, which rewrites as $[\Phi_{t+T^* \circ \tau_{x,t}}^{t-}(\xi_1)](x) = [\Phi_{t+T^* \circ \tau_{x,t}}^{t-}(\xi_2)](x)$. 
We have seen that $T^* \leq t+T^* \circ \tau_{x,t}$, so we can deduce that $e(\alpha, T^*, \xi)$ does not depend on   $\xi \in \S^{\Z^d}$. 
\end{proof}

Let $M$ denote the intensity measure of the set $\Amb_1$, i.e. the positive measure on $\Z^d \times \R$ defined for all Borel set $C$ by  
$$M(C) := \E ( |\Amb_1 \cap C| )  =  \E \left(  \sum_{(z,i,s) \in  H }  \sum_{y \in A_i} \un( (z+y,s) \in C ) \right).$$
We use the notation $\star$ for  the convolution product of measures on $\Z^d \times \R$. For all $n \geq 0$,  $M^{\star n}$ denotes the  convolution product $M \star \cdots \star M$ with $n$ factors, with the convention $M^{\star 0} := \delta_{(0,0)}$.

Our key first-moment estimates are given in the next proposition and its corollary.

\begin{prop}\label{p:branch-ambig}
For any measurable $f \ : \ \Z^d \times \R \to \R_+$, and any $n \geq 0$, one has
$$\E\left( \sum_{\zeta \in \Amb_n} f(\zeta) \right)  \leq   \int f(\zeta) d M^{\star n}(\zeta),$$
with the convention $M^{\star 0} := \delta_{(0,0)}$.
\end{prop}

\begin{coroll}\label{c:branch-ambig}
For measurable $f \ : \ \Z^d \times \R \to \R_+$, any non-negative $\F_0-$measurable random variable $D$, and any $n \geq 0$, one has
$$\E\left( \sum_{\zeta \in \Amb_n} f(\zeta) \cdot D \circ \tau_{\zeta} \right)  \leq \E(D)  \cdot  \int f(\zeta) d M^{\star n}(\zeta),$$
with the convention $M^{\star 0} := \delta_{(0,0)}$.
\end{coroll}

The proof makes use of the so-called refined Campbell theorem (see \cite{StoKenMec}), which we quote here in the special form we need:
\begin{theorem}\label{t:Campbell}
For any  measurable map $\Psi \ : \ (\Z^d \times \I \times \R) \times \Omega \to \R_+$, one has the following identity:
 $$\E \left( \sum_{\alpha \in \Nsp}   \Psi(\alpha, \Nsp)       \right) = \int  \E(  \Psi(\alpha, \Nsp \cup \{ \alpha \})            )       dJ(\alpha),$$
 where $J$ denotes the intensity measure of $\Nsp$.
\end{theorem} 

A crucial property of the sets $\Amb_n$ is that they possess the stopping property, as stated in the following lemma.
\begin{lemma}\label{l:mesurabil-Amb}
For all $n \geq 0$, $\Amb_n$ has the stopping property, i.e. for all $t<0$, $\Amb_n \cap \Nsp_t$ is $\F_t-$measurable.
\end{lemma}

Before we prove Lemma \ref{l:mesurabil-Amb}, we need the following lemma.
\begin{lemma}\label{l:progres}
There exists a measurable map $\Xi$ from $\tilde{\Omega} \times ]-\infty,0[$ to $\tilde{\Omega}$ such that, for all $t <0$, 
$$H \cap \Nsp_t  = \Xi(t, \Nsp_t).$$
\end{lemma}

\begin{proof}
We first prove that there exists a measurable map $V$ from $\Omega \times  ]-\infty,0[$ to  $\tilde{\Omega}$ such that, for all $t >0$, 
\begin{equation}\label{e:mesurabilite-infini}H \cap \Nsp_t  = V(t, \Nsp).\end{equation}
For each $(x,i) \in \Z^d \times \I$, let $(\psi(x,i,k))_{k \geq 1}$ denote the successive points $ \Nsp$ whose coordinate on $\Z^d \times \I$ is $(x,i)$, in decreasing order of the $\R-$coordinate. We let $\psi(x,i,k) =: (x,i, s(x,i,k))$. Given a Borel set $B$ of $ \Z^d \times \I \times \R$, one has that, for all $t<0$,   
$$| H \cap \Nsp_t  \cap B| = \sum_{x,i,k} \un(\psi(x,i,k) \in H) \un(s(x,i,k) \geq t) \un( \psi(x,i,k) \in B).$$
Since $H$ is a measurable map from $(\Omega,\F)$ to $(\tilde{\Omega}, \tilde{F})$, one has that 
$\omega \mapsto   \un(\psi(x,i,k) \in H)(\omega)$ is a measurable map from $(\Omega, \F)$ to $\R$.
This is also the case for $\omega \mapsto   \un(\psi(x,i,k) \in B)(\omega)$. Finally,  $(t, \omega) \mapsto (s(x,i,k) - t)$ is measurable from  
$\Omega \times  ]-\infty,0[$ to $\R$, so this is also the case for $ \un(s(x,i,k) \geq t)$. We conclude that $(\omega, t) \mapsto  H \cap \Nsp_t $ is 
measurable from $\Omega \times  ]-\infty,0[$ to  $\tilde{\Omega}$, whence the existence of $V$. 
Consider now an arbitrarily fixed element $\omega_0 \in \Omega$ that contains no point with $0$ $\R-$coordinate, and define the map  $\mathfrak{a}$ from $\tilde{\Omega} \times   ]-\infty,0[$ to $\Omega$  by
$$\mathfrak{a}(\tilde{\omega},t) := \left( \tilde{\omega} \cap (\Z^d \times \I \times [t,0[) \right) \cup \tau_{0,t}(\omega_0).$$
One checks that $\mathfrak{a}$ is measurable by writing$$| \mathfrak{a}(\omega,t)  \cap B     |  =  \sum_{x,i,k} \un(s(x,i,k)(\tilde{\omega}) \geq t) \un( \psi(x,i,k)(\tilde{\omega}) \in B)+|\tau_{0,t}(\omega_0) \cap B|,$$
where we have extended the definition of $\psi(x,i,k)$ to $\tilde{\Omega}$ in the obvious way, with the convention that $s(x,i,k)$  takes the value $-\infty$ when the value of $k$ excesses the number of points to be indexed.  Now, since, for any given $t<0$, $H \cap \Nsp_t $ is $\F_t-$measurable by assumption, one has that, in view of \eqref{e:mesurabilite-infini}, 
for any $t<0$, 
$$ V( t, \Nsp   ) = V(t, \mathfrak{a}( \Nsp_t, t )).$$
As a consequence, the conclusion of the proposition is achieved  by defining 
$$\Xi(t, \tilde{\omega}) :=   V(t, \mathfrak{a}( \tilde{\omega}, t )).$$
  \end{proof}

\begin{proof}[Proof of Lemma \ref{l:mesurabil-Amb}]
We re-use the notations introduced in the proof of Lemma \ref{l:progres}.
The proof is by induction. For $n=0$, the result is obvious since $\Amb_0 := \{ (0,0) \}$.
Now assume the result to be true for a given $n \geq 0$. 
Define $D_{n}:= \{   (w,i,k) ;        \  \pi(\psi(w,i,k)) \cap \Amb_{n} \neq \emptyset           \}$. 
By definition, one has
\begin{equation}\label{e:redef-Amb}\Amb_{n+1} :=   \bigcup_{(x,i,k) \in D_n}  \pi \left( \tau_{x,s(x,i,k)}^{-1} \circ H \circ \tau_{x,s(x,i,k)} \right),\end{equation}
with the slight abuse of notation that, given a subset $C$ of $\Z^d \times \I \times \R$,  $\pi(C) := \bigcup_{c \in C} \pi(c)$.     
 
Now consider $t<0$. From Lemma \ref{l:progres}, we deduce that the map on $\Omega \times [t,0]$ defined by
$(\omega, s) \mapsto  \left( \tau_{x,s}^{-1} \circ H \circ \tau_{x,s} (\omega) \right) \cap \Nsp_t(\omega)$
is $\F_t \otimes \mathcal{B}([t,0])-$measurable. On the other hand, our induction hypothesis shows that for any $(x,i,k)$, 
the event $\{ (x,i,k) \in D_n ,    \  s(x,i,k) \geq t  \}$ is $\F_t-$measurable. We can then deduce from \eqref{e:redef-Amb} that 
$\Amb_{n+1} \cap \Nsp_t$ is $\F_t-$measurable.
\end{proof}

\begin{proof}[Proof of Proposition \ref{p:branch-ambig}]
We proceed by induction. For $n=0$, the result is immediate since by definition $\Amb_0 := \{ (0,0) \}$, while, for $n=1$, the result is a direct consequence of $M$ being the intensity measure of $\Amb_1$. 
Now consider $n \geq 1$, and note that, by definition, 
\begin{equation}\label{e:contrib-parents}\sum_{\zeta \in \Amb_{n+1} } f(\zeta)   \leq \sum_{(x,t) \in \Amb_n}  g_{x,t},\end{equation}
with 
$$g_{x,t} := \sum_{(z,i,s) \in \tau_{x,t}^{-1} \circ H \circ \tau_{x,t}} \sum_{y \in A_i}  f(z+y,i,s).$$

For $(x,i,t) \in \Z^d \times \I \times \R$, define $\pi(x,i,t) := \bigcup_{z \in A_i} \{ (x+z,t) \}$. Now we rewrite 
$$  \sum_{(x,t) \in \Amb_n}  g_{x,t}  = \sum_{\alpha \in \Nsp}   \Psi( \alpha, \Nsp  ), $$
with $$\Psi(\alpha, \omega) := \un (\pi(\alpha)\subset \Amb_n(\omega))  \sum_{(x,t) \in \pi(\alpha)}  g_{x,t}(\omega).$$
Applying Campbell's theorem (Theorem \ref{t:Campbell}), we deduce that 
\begin{equation}\label{e:Campbell-1} \E \left( \sum_{\zeta \in \Amb_{n+1} } f(\zeta) \right) \leq   \int  \E(  \Psi(\alpha, \Nsp \cup \{ \alpha \})            )       dJ(\alpha).\end{equation}
No we deduce from Lemma \ref{l:mesurabil-Amb} that, for all $\alpha=(w,j,t) \in \Z^d \times \I \times \R$ such that $t<0$, the event 
$\{ \pi(\alpha) \subset \Amb_n(\Nsp \cup \{ \alpha \}) \}$ is $\F_t-$measurable. On the other hand, for all $(x,t) \in \pi(\alpha)$, 
the random variable $g_{x,t}$ is measurable with respect to $\sigma(\Nsp_{<t})$, where 
$\Nsp_{<t} := \Nsp \cap (\Z^d \times \I \times ]-\infty, t [)$. As a consequence, $ \un (\pi(\alpha)\subset \Amb_n(\Nsp \cup \{ \alpha \}) $ and $ \sum_{(x,t) \in \pi(\alpha)}  g_{x,t}(\Nsp \cup \{ \alpha \})$ are independent. Moreover, one has that 
$$\E \left(   \sum_{(x,t) \in \pi(\alpha)}  g_{x,t}(\Nsp  \cup \{ \alpha \} )  \right) = \sum_{(x,t) \in \pi(\alpha)} \varphi(x,t),$$ with 
$$\varphi(x,t) :=  \int f(\zeta) d (\delta_{(x,t)} \star M)(\zeta).$$ 
Thus
\begin{equation} \label{e:indep-temps} \E(  \Psi(\alpha, \Nsp \cup \{ \alpha \})            )   =  \E \left(  \un(  \pi(\alpha) \subset \Amb_n(\Nsp \cup \{ \alpha \})  ) \times  \sum_{(x,t) \in \pi(\alpha)}    \varphi(x,t) \right) ,\end{equation}
Applying again Campbell's theorem, 
we deduce from \eqref{e:indep-temps} that 
\begin{eqnarray*} \int  \E(  \Psi(\alpha, \Nsp \cup \{ \alpha \})            )       dJ(\alpha) &=& \E \left(      \sum_{\alpha \in \Nsp}     \un(  \pi(\alpha) \subset \Amb_n)   \sum_{(x,t) \in \pi(\alpha)}    \varphi(x,t)         \right)   \\ &=& \E \left( \sum_{\zeta \in \Amb_n}  \varphi(\zeta) \right).\end{eqnarray*}
Assuming the conclusion of the proposition to be true for $n$, we deduce that 
$$ \int  \E(  \Psi(\alpha, \Nsp \cup \{ \alpha \})            )       dJ(\alpha)   =  \int \varphi(\zeta) d M^{\star n}(\zeta) = \int f(\zeta) d M^{\star (n+1)}(\zeta).$$
In view of \eqref{e:Campbell-1}, this establishes the conclusion of the proposition for $n+1$.
\end{proof}

\begin{proof}[Proof of Corollary \ref{c:branch-ambig}]
For $n=0$ the result is immediate. For $n \geq 1$, 
$$\sum_{\zeta \in \Amb_n} f(\zeta) \cdot D \circ \tau_{\zeta} = \sum_{\alpha \in \Nsp}   \Psi'(\alpha, \Nsp),$$
with
$$  \Psi'(\alpha,\omega) :=  \un(\pi(\alpha) \subset \Amb_n(\omega)) \sum_{(x,t) \in \pi(\alpha)} g'_{x,t}(\omega),$$
$$g'_{x,t} :=  f(x,t)  \cdot  D \circ \tau_{x,t}(\omega).$$
Then Campbell's theorem  shows that  
$$\E\left( \sum_{\zeta \in \Amb_n} f(\zeta) \cdot D \circ \tau_{\zeta} \right)   =     \int  \E(  \Psi'(\alpha, \Nsp \cup \{ \alpha \})            )       dJ(\alpha).$$
As in the proof of Proposition \ref{p:branch-ambig}, given
$\alpha=(w,j,t) \in \Z^d \times \I \times \R$ such that $t<0$, the event 
$\{ \pi(\alpha) \subset \Amb_n(\Nsp \cup \{ \alpha \}) \}$ is $\F_t-$measurable while, for all $(x,t) \in \pi(\alpha)$, 
the random variable $g'_{x,t}$ is measurable with respect to $\sigma(\Nsp_{<t})$.
Moreover, $$\E \left( g'_{x,t} (\Nsp \cup \{ \alpha \} \right)) =  f(x,t) \cdot  \E(D),$$
so that 
$$ \E(  \Psi'(\alpha, \Nsp \cup \{ \alpha \})            )   = \E \left(  \un(\pi(\alpha) \subset \Amb_n(\omega)) \sum_{(x,t) \in \pi(\alpha)} f(x,t) \cdot  \E(D)    \right).$$
Another application of Campbell's theorem yields that 
\begin{eqnarray*} \int  \E(  \Psi'(\alpha, \Nsp \cup \{ \alpha \})            )  dJ(\alpha)
&=&  \E(D) \cdot    \E \left(      \sum_{\alpha \in \Nsp}     \un(  \pi(\alpha) \subset \Amb_n)   \sum_{(x,t) \in \pi(\alpha)}    f(x,t)     \right)   \\ &=& \E(D)  \cdot \E \left( \sum_{\zeta \in \Amb_n}  f(\zeta) \right).\end{eqnarray*}
 Proposition \ref{p:branch-ambig} then yields the conclusion.
\end{proof}

\begin{proof}[Proof of Theorem \ref{t:finitude}]
Assume that $\g < 1$,  and note that, by definition, one has 
$\int d M(\zeta) = \g$. We now use Proposition \ref{p:branch-ambig} with $f \equiv 1$, and obtain that, for all $n \geq 0$, $\E(|\Amb_n|) \leq \g^n$. 
As a consequence, 
$$\E(|\Amb_{\infty}|) \leq \E \left( \sum_{n \geq 0}  |\Amb_n| \right) = \sum_{n \geq 0} \E \left(   |\Amb_n| \right) \leq \sum_{n \geq 0} \g^n < +\infty.$$
It is now clear that $\P (| \Amb_{\infty}|<+\infty)=1$.
Similarly, applying Corollary \ref{c:branch-ambig} with $f \equiv 1$ and $D := \un(T=-\infty)$ yields that, for all $n \geq 0$, 
$$\E \left( \sum_{(x,t) \in \Amb_n }     \un(T \circ \tau_{x,t} = -\infty) \right)  = 0.$$
As a consequence,  with probability one, 
 $T \circ \tau_{x,t} > -\infty$ for all $(x,t) \in \Amb_{\infty}$.  
We have thus proved that $\P(T^* > -\infty) =1$. The conclusion of the theorem is now a consequence of Proposition \ref{p:base-ambig}.
\end{proof}

\begin{proof}[Proof of Theorem \ref{t:temps-exp}]
We apply  Corollary \ref{c:branch-ambig} with $f(x,t) \equiv \exp(\lambda t)$ and $D := \exp(\lambda T)$.
As a result, for all $n \geq 0$, 
$$\E \left( \sum_{  (x,t) \in \Amb_n }     e^{\lambda (t+ T \circ \tau_{x,t})}          \right) \leq \E\left( e^{\lambda T} \right)  \left( \int  e^{\lambda t} dM^{\star n}(x,t)  \right).$$ 
One has $$\int  e^{\lambda t} dM^{\star n}(x,t) =    \left( \int  e^{\lambda t} dM(x,t)  \right)^n = (\Lambda_{H,time}(\lambda))^n.$$
Summing over $n \geq 0$, we obtain that 
$$\E \left( \sum_{  (x,t) \in \Amb_{\infty} }     e^{\lambda (t+ T \circ \tau_{x,t})}          \right)  \leq \sum_{n \geq 0} \E\left( e^{\lambda T} \right)  (\Lambda_{H,time}(\lambda))^n = \frac{ \E\left( e^{\lambda T} \right) }{1 - \Lambda_{H,time}(\lambda)}.$$
Now by definition of $T^*$, using the fact that $\lambda < 0$, $$\E(\exp(\lambda T^*)) \leq \E \left( \sum_{  (x,t) \in \Amb_{\infty} }     e^{\lambda (t+ T \circ \tau_{x,t})}          \right).$$
The conclusion follows.
\end{proof}

\begin{proof}[Proof of Theorem \ref{t:espace-exp} (sketch)]
First note that, when $\g<1$, $L^*_+$ and $L^*_-$ are a.s. finite, using an argument similar to the one establishing that $T^*$ is a.s. finite in the proof of Theorem \ref{t:finitude}. Moreover, the proof of the estimates on $\E(\exp(\lambda L^*_+))$ and $\E(\exp(-\lambda L^*_-))$ is completely similar to the proof of Theorem \ref{t:temps-exp}. It remains to prove that $L^*$ indeed bounds the width of $T^*$. This is done by adapting the proof of Proposition \ref{p:base-ambig} as follows. We work on the event $T^*>-\infty$.  First note that, thanks to the fact that 
$L$ defines a stopping box, $L$ is $\F^{-L', L'}-$measurable for any  random variable such that 
$L \leq L'$. Then observe that  $\Amb_{\infty}$ is measurable with respect to $\F^{-L^*,L^*}$, since, for each $n \geq 0$ and $(x,t) \in \Amb_{n}$, one has that   
$L^*_- \leq    x - L  \circ \tau_{x,t} \leq       x + L \circ \tau_{x,t}   \leq L^*_+$. We now start the induction with the case $|\Amb_{\infty}|=1$. Then the values of the $e(\alpha, T, \xi)$ for $\alpha \in H$ are completely determined by $H$ itself, and $H$ is measurable with respect to $\F^{-L, L}$, so we are done. If   $|\Amb_{\infty}|=n+1$, we apply the induction hypothesis to every  $|\Amb_{\infty} \circ \tau_{x,t}|$ such that $x=z+y$ for some $\alpha=(z,i,t) \in \Amb_1$ and $y \in A_i$, then deduce that $[\Phi_{T^*}^{0-}(\xi)](0)$ has the required measurability properties. 
\end{proof}

\section{Proof of Theorem \ref{t:theoreme-principal}}\label{s:preuve-2}

We start with a proposition showing that $(T,H)$ is a CFTP time with ambiguities for the perturbed dynamics if $\XX_{\infty}$ is a finite set $\P-$a.s.

\begin{prop}\label{p:base-perturb}
If $\P(|\XX_{\infty}|<+\infty)=1$, then $(T, H)$ is a CFTP time with ambiguities.
\end{prop}

\begin{proof}
First note that the stopping property of $\XX_{\infty}$ is a direct consequence of the way the process is constructed. We now work on the event that $|\XX_{\infty}|$ is finite. Consider $\xi \in \S^{\Z^d}$, and let $\Qsp$ be the element of $\Omega^u$ obtained from $\Nsp$ by replacing any 
$\alpha =(x,i,s) \in \Nsp$ such that $s \geq T$ and $i \in \I^p$ by $(x,\iota_{e(\alpha,T,\xi)},s)$, and suppressing  any   $\alpha =(x,i,s) \in \Nsp$ such that $s < T$ and $i \in \I^p$.  Let $h_t(\xi)$ denote the random variable defined on $\Omega^u$ by $$h_t(\xi) := [(\Phi^u)_{t}^{0-}(\xi)](0),$$
where $\Phi^u$ denotes the stochastic flow defined by $\Nsp^u$ on $\Omega^u$. From the definition of the dynamics, we see that 
\begin{equation}\label{e:premiere-substitution}[\Phi_T^{0-}(\xi)](0) = [h_T(\xi)](\Qsp).\end{equation} Now define a path $c_0,\ldots, c_m$ in $\TT$ as follows. Start with $c_0:=r$. Then assume that $c_0,\ldots, c_k$ have been defined. If $c_k$ has no son in $\TT$, the path ends at $c_k$, so that $m:=k$. If $c_k$ has a single son in $\TT$, then $c_{k+1}$ is defined to be this single son. Finally, if $c_k$ has several sons in $\TT$, $c_{k+1}$ is defined to be the son associated with the value $v:=e((x_{c_k}, i_{c_k}, t_{c_k}),T,\xi)$. 
By definition of the exploration processes $\X^u$ and $\XX$, one has that $\X^u_{\infty}(\Qsp) =  \XX_{c_m}$. Then by definition $T^u(\Qsp) = \inf \{ t;   \  (x,i,t)  \in    \X^u_{\infty}(\Qsp)      \}$, so that the identity $\X^u_{\infty}(\Qsp) =  \XX_{c_m}$ implies that $T^u(\Qsp) \geq T$. From the fact that $T^u$ is associated with the exploration process specified by $\theta$, there exists a measurable map $G$ defined on the set $\Omega^u_f$ such that, on the event $\{ T^u > -\infty \} = \{  |\X^u_{\infty}|<+\infty \}$, one has, for every $\chi \in \S^{\Z^d}$,  $h_{T^u}(\chi) = G(\X^u_{\infty})$.
As a consequence,  
 
\begin{equation}\label{e:seconde-substitution}[h_T(\xi)](\Qsp) = G( \XX_{c_m} ).\end{equation} 
It is now immediate from \eqref{e:premiere-substitution} and \eqref{e:seconde-substitution} that 
if $e(\alpha, \xi_1, T) = e(\alpha, \xi_2, T)$ for every $\alpha \in H$,  then $[\Phi_T^{0-}(\xi_1)](0) = [\Phi_T^{0-}(\xi_2)](0)$, since both values of $ \XX_{c_m}$ obtained 
starting from $\xi_1$ or $\xi_2$ are identical.
\end{proof}

The next two propositions are the key first-moment estimates needed to control $\XX_{\infty}$.

Define a kernel $K$ on $\Omega^u_f$ as follows. 
If $\theta(X) = \emptyset$, then $K(X, \cdot) = \delta_{X}(\cdot)$. 
If $\theta(X) \neq \emptyset$, let $s :=  \inf \{ t;   \   (x,i,t) \in X  \}$, and 
let $$d K(X, X \cup \{ (x,i,t)     \})   =   r_i \exp(|\theta(X) | \r_u (t-s))   \un(t < s) \un(x \in \theta(X)) dJ^u(x,i,t),$$ 
with $\r_u := \sum_{j \in \I^u} r_j$. From the definition, one has the following.
\begin{prop}\label{p:Markov} The sequence $(\X^u_{\ell})_{\ell \geq 0}$ is a Markov chain on $\Omega^u_f$ with initial state $\emptyset$ and transition kernel $K$. 
\end{prop}

Now define a kernel $\hat{K}$ on  the set of elements $X \in \Omega^u_f$ such that $\theta(X) \neq \emptyset$ as follows.  
Let $s := \inf \{ t;   \   (x,i,t) \in X  \}$, and 
let $$d \hat{K}(X, X \cup \{ (x,i,t)     \})   =   \hat{r}_i \exp(|\theta(X) | \r (t-s))  \un(t < s) \un(x \in \theta(X)) dJ^u(x,i,t),$$
with, for $i \in \I^u$,  
$\hat{r}_i :=  r_i +   \sum_{j \in \I^p } \sum_{ v \in f_j(A_j)  } r_j \un(\iota_v = i)$,       
and
$\r := \sum_{j \in \I} r_j$.
Define also the kernel $L$ by 
$$dL(X, (x,i,t)) =   r_i \exp(|\theta(X) | \r (s-t))  \un(t < s) \un(x \in \theta(X))  dJ(x,i,t).$$

For $\ell \geq 0$, let $\TT_{\ell}$ (resp. $\TT'_{\ell}$) denote the set of vertices at distance $\ell$ from the root in $\TT$ (resp. $\TT'$); let also 
$$\Gamma_{\ell} := \left\{  (X_0,\ldots, X_{\ell}) \in (\Omega^u_f)^{\ell+1};   \  \theta(X_0) \neq \emptyset, \ldots,        \theta(X_{\ell-1}) \neq \emptyset      \right\},$$
$$\Delta_{\ell} := \left\{  (X_0,\ldots, X_{\ell}) \in (\Omega^u_f)^{\ell+1};   \  \theta(X_0) \neq \emptyset, \ldots,        \theta(X_{\ell}) \neq \emptyset  \right\}.$$

\begin{prop}\label{p:calcul-branchant}
For every $\ell \geq 0$, and any measurable map $F \ : \ \Omega^u_f \to \R_+$, one has the identity
$$\E\left( \sum_{a \in \TT_{\ell}} F(\XX_a)\right) =  \int_{\Gamma_{\ell}}  F(X_{\ell}) d \delta_{\emptyset}(X_0) d\hat{K}(X_0,X_1) \cdots  d\hat{K}(X_{\ell-1},X_{\ell}).$$
For every $\ell \geq 0$, and any measurable map $f \ : \   \Z^d \times \I \times \R  \to \R_+$, one has the identity
$$\E\left( \sum_{a \in \TT'_{\ell}} f(x_a,i_a,t_a)\right) = \int f(\alpha) d \delta_{\emptyset}(X_0) d\hat{K}(X_0,X_1) \cdots  d\hat{K}(X_{\ell-1},X_{\ell}) dL(X_{\ell}, \alpha),$$
where the integral is over $(X_0,\ldots, X_{\ell}, \alpha) \in \Delta_{\ell} \times (\Z^d \times \I \times \R)$.
\end{prop}

\begin{proof}[Proof]
The proof is similar to that of Proposition \ref{p:branch-ambig}. 
We start with the first identity, whose proof is by induction. For $\ell=0,1$, the result is a direct consequence of the definition.
Now for $\ell \geq 1$, write
\begin{equation}  \sum_{a \in \TT_{\ell+1}} F(\XX_a) = \sum_{\alpha \in \Nsp}  Z_1(\alpha, \Nsp) ,\end{equation}
where
$$Z_1(\alpha, \cdot) := \sum_{c \in \TT'_{\ell-1}}  \un(\alpha = (x_c,i_c,t_c))  \sum_{b \in \TT_{\ell},  \ b \leftarrow c}      \un(\theta(\XX_b) \neq \emptyset)    Z_2(\XX_b, \cdot),$$
with $b \leftarrow c$ meaning that $b$ is a son of $c$ in $\TT$, and  with
$$ Z_2(X, \cdot):=  \un(j \in \I^u) F(X \cup \{  (y, j, s)    \} ) +     \un(j \in \I^p)    \sum_{ v \in f_{j}(A_{j})  }    F(X \cup \{  (y, \iota_v, s)    \} ),$$
with $(y,j,s) := \n(\theta(X), t),$ and $t:= \inf \{  \nu;   \  (w,k,\nu) \in X    \}$. (We write $Z_1(\alpha, \cdot)$ and $Z(X,\cdot)$ to make the dependence on $\omega$ explicit.)

By Campbell's theorem (Theorem \ref{t:Campbell}), one has that 
$$\E \left(     \sum_{a \in \TT_{\ell+1}} F(\XX_a)    \right)  = \int \E(   Z_1(\alpha, \Nsp \cup \{  \alpha \})) d J(\alpha).$$
Now given $\alpha = (x,i,t)$, define the random finite counting measure\footnote{We equip the set of finite counting measures on $\Omega^u_f$ with the $\sigma-$algebra generated by all the maps of the form $Y \mapsto Y(B)$, where $B$ belongs to the $\sigma-$algebra defined on $\Omega^u_f$.} $\mathcal{M}$ on  $\Omega^u_f$ by
$$  \mathcal{M} :=   \sum_{c \in \TT'_{\ell-1}}  \un(\alpha = (x_c,i_c,t_c))  \sum_{b \in \TT_{\ell},  \ b \leftarrow c}      \un(\theta(\XX_b) \neq \emptyset)    \delta_{\XX_b},$$
so that $$Z_1(\alpha, \cdot) = \int Z_2(X, \cdot) d \mathcal{M}(X, \cdot).$$
 Note that $\mathcal{M}$ is $\F_t-$measurable, while, for any $X \in \Omega^u_f$ such that $\theta(X) \neq \emptyset$ and $t:= \inf \{  \nu;   \  (w,k,\nu) \in X    \}$, $Z_2(X ,\cdot)$ is measurable with respect to $\sigma(\Nsp_{<t})$ and satisfies $\E(Z_2(X,\cdot)) = \hat{K}F(X)$.
 We deduce that  
$$\E( Z_1(\alpha, \Nsp \cup \{ \alpha \})     ) =  \E \left(   \int    \hat{K}F(X)  d \mathcal{M}(X, \Nsp \cup \{  \alpha \}) \right) = \E(Z_3(\alpha, \Nsp \cup \{ \alpha \})),$$
where 
$$Z_3(\alpha, \cdot) :=   \sum_{c \in \TT'_{\ell-1}}  \un(\alpha = (x_c,i_c,t_c))  \sum_{b \in \TT_{\ell},  \ b \leftarrow c}   \hat{K} F( \XX_b ).$$
Using Campbell's theorem again shows that 
$$\int \E(Z_3(\alpha, \Nsp \cup \{ \alpha \}))dJ(\alpha) = \E \left(    \sum_{b \in \TT_{\ell}}      \hat{K} F( \XX_b ) \right).$$
This computation allows induction over $\ell$ to be used to prove the desired identity for all $\ell \geq 1$.
The second identity of the Proposition can be deduced from the first one, using an argument similar to the derivation of Corollary \ref{c:branch-ambig} from Proposition \ref{p:branch-ambig}. 
\end{proof}

Combined with Propositions \ref{p:Markov} and \ref{p:calcul-branchant}, the following remark is  the key to obtaining estimates on $\XX$ from the properties of $\X^u$.  Consider $X \in \Omega^u_f$ such that $\theta(X) \neq \emptyset$. From the definition of $\epsilon$ given in \eqref{e:def-epsilon} and the fact that $\r_u \leq \r$, one has that
\begin{equation}\label{e:compare-densite}d\hat{K}(X, X \cup \{ (x,i,t)     \}) \leq    (1+\epsilon)  d K (X, X \cup \{ (x,i,t)     \}).\end{equation}

We can now prove the various estimates that are needed in the proof of Theorem \ref{t:theoreme-principal}.

\begin{lemma}\label{l:estimation-1}
For every $\ell \geq 1$, 
$$\E(|\TT_{\ell}|) \leq (1+\epsilon)^{\ell}  \P^u( |\X^u_{\infty}| \geq \ell  ).$$
\end{lemma}

\begin{proof}
By Proposition \ref{p:calcul-branchant}, one has that
$$\E(|\TT_{\ell}| ) = \int_{(X_0,\ldots, X_{\ell}) \in \Gamma_{\ell}} d \delta_{\emptyset}(X_0) d\hat{K}(X_0,X_1) \cdots  d\hat{K}(X_{\ell-1},X_{\ell}),$$
From \eqref{e:compare-densite}, one deduces that
$$\E(|\TT_{\ell}|) \leq (1+\epsilon)^{\ell} \int_{(X_0,\ldots, X_{\ell}) \in \Gamma_{\ell}} d \delta_{\emptyset}(X_0) dK(X_0,X_1) \cdots  d K(X_{\ell-1},X_{\ell}).$$
But by Proposition \ref{p:Markov},
$$\int_{(X_0,\ldots, X_{\ell}) \in \Gamma_{\ell}} d \delta_{\emptyset}(X_0) dK(X_0,X_1) \cdots  d K(X_{\ell-1},X_{\ell}) = \P^u( |\X^u_{\infty}| \geq \ell  ),$$
so that 
$$\E(|\TT_{\ell}|) \leq (1+\epsilon)^{\ell}  \P^u( |\X^u_{\infty}| \geq \ell  ).$$
\end{proof}

\begin{lemma}\label{l:estimation-2}
For every $\ell \geq 0$, 
$$\E \left( \sum_{a \in \TT'_{\ell}} |A_{i_a}| \un(i_a \in \I^p) \right) \leq \kappa (1+\epsilon)^{\ell}   \P^u( |\X^u_{\infty}| > \ell  ).$$
\end{lemma}

\begin{proof}
Let $f(x,i,t):= |A_{i}| \un(i \in \I^p)$. 
Note that, given $X \in \Omega^u_f$ such that $\theta(X) \neq \emptyset$, one has 
$\int_{(x,i,t) \in   \Z^d \times \I \times \R} f(x,i,t) dL(X,(x,i,t)) =  \kappa$. Using Proposition  \ref{p:calcul-branchant} and \eqref{e:compare-densite}, 
one obtains that 
$$\E \left( \sum_{a \in \TT'_{\ell}} |A_{i_a}| \un(i_a \in \I^p) \right) \leq \kappa (1+\epsilon)^{\ell} I_2,$$
where 
$$I_2=\int_{(X_0,\cdots, X_{\ell}) \in \Delta_{\ell}} d \delta_{\emptyset}(X_0) d K(X_0,X_1) \cdots  d K(X_{\ell-1},X_{\ell})=  \P^u( |\X^u_{\infty}| > \ell).$$
 \end{proof}

\begin{lemma}\label{l:estimation-3}
For every $\ell \geq 0$, and $\lambda < \r$, 
$$  \E \left( \sum_{a \in \TT'_{\ell}}  \exp(\lambda t_a)  \right) \leq   \frac{\r}{\r  - \lambda} (1+\epsilon)^{\ell}   \E^u(\exp(\lambda t_{\ell-1}) \un(|\X^u_{\infty}| > \ell )),$$
with the convention $t_{-1} := 0$ (remember that $(x_n, i_n, t_n)_{n \geq 0}$ corresponds to the exploration process $\X^u$).
\end{lemma}

\begin{lemma}\label{l:estimation-4}
For every $\ell \geq 0$, and $\lambda < \r$, 
$$  \E \left( \sum_{a \in \TT'_{\ell}}  |A_i| \exp(\lambda t_a) \un(i_a \in \I^p) \right) \leq  \frac{\kappa \r}{\r  - \lambda} (1+\epsilon)^{\ell}   \E^u(\exp(\lambda t_{\ell-1}) \un(|\X^u_{\infty}| > \ell )),$$
with the convention $t_{-1} := 0$.
\end{lemma}

\begin{proof}
We prove Lemma \ref{l:estimation-4}, the proof of Lemma \ref{l:estimation-3} being quite similar.
Let $f(x,i,t):= |A_i| \exp(\lambda t) \un(i \in \I^p)$. Consider $X \in \Omega^u_f$ such that $\theta(X) \neq \emptyset$, and let  $s := \inf \{ t;   \   (x,i,t) \in X  \}$.
One has 
$\int_{(x,i,t) \in   \Z^d \times \I \times \R} f(x,i,t) dL(X,(x,i,t)) = \kappa  \r | \theta(X)| ( \r |\theta(X)| - \lambda )^{-1} \exp( \lambda s) \leq  \kappa \r (\r  - \lambda)^{-1}  \exp(\lambda s),$
since $|\theta(X)| \geq 1$.
Using Proposition  \ref{p:calcul-branchant} and \eqref{e:compare-densite}, we deduce that 
$$  \E \left( \sum_{a \in \TT'_{\ell}} \exp(\lambda t_a) \un(i_a \in \I^p) \right) \leq \kappa\r  (\r  - \lambda)^{-1} (1+\epsilon)^{\ell} I_3,$$
with
$$I_3=\int_{(X_0,\cdots, X_{\ell}) \in \Delta_{\ell}}  \exp(\lambda s_{\ell-1}) d \delta_{\emptyset}(X_0) d K(X_0,X_1) \cdots  d K(X_{\ell-1},X_{\ell})$$
and $s_{\ell-1} :=  \inf \{ t;   \   (x,i,t) \in X_{\ell}  \}$, with the convention $s_{\ell-1}:=0$.
Then note that $I_3 = \E^u(\exp(\lambda t_{\ell-1}) \un(|\X^u_{\infty}| > \ell ))$. 
\end{proof}

\begin{lemma}\label{l:estimation-6}
For every $\ell \geq 1$, $\lambda \in \R$ and $1 \leq q \leq d$,  
$$ \E \left( \sum_{a \in \TT'_{\ell}} \sum_{z \in A_{i_a}} e^{\lambda (x_a+z)_q} \un\mbox{\tiny $(i_a \in \I^p)$} \right) \leq Q (1+\epsilon)^{\ell}   \E^u \left( \frac{1}{|\theta(\X^u_{\ell})|}  \sum_{x \in \theta(\X^u_{\ell})} e^{\lambda x_q} \un\mbox{\tiny $(|\X^u_{\infty}| > \ell )$} \right),$$
where 
\begin{equation}Q := \left( \sum_{i \in \I^p} (r_i/\r) \sum_{z \in A_i} \exp(\lambda z_q)  \right).\end{equation}

\end{lemma}

We are now ready to prove Theorem \ref{t:theoreme-principal}. From  now on, we assume that there exists $\mu>0$ such that 
$\E^u\left( e^{\mu | \X^u_{\infty}} | \right) < +\infty$. As a consequence, there exists a finite constant $C$ such that, for all $\ell \geq 1$, 
\begin{equation}\label{e:queue-exp} \P^u(|\X^u_{\infty}| \geq \ell) \leq C \exp(-\mu \ell).\end{equation}

We first prove that, as soon as $\epsilon < \mu$, $\E( |\XX_{\infty}|)<+\infty$. Indeed, one has that $|\XX_{\infty}| \leq |\TT|$, so that 
\begin{equation}\label{e:somme-tranches}\E(|\XX_{\infty}|) \leq \E(|\TT|) =  \E \left( \sum_{\ell \geq 0} |\TT_{\ell}| \right)= \sum_{\ell \geq 0} \E(|\TT_{\ell}|).\end{equation}
By Lemma \ref{l:estimation-1}, for all $\ell \geq 1$, one has $\E(|\TT_{\ell}|) \leq (1+\epsilon)^{\ell}  \P^u( |\X^u_{\infty}| \geq \ell  )$. 
Combining \eqref{e:queue-exp} and \eqref{e:somme-tranches}, we see that  $\E( |\XX_{\infty}|)<+\infty$ when $\epsilon < \mu$. By Proposition \ref{p:base-perturb}, this proves that the pair $(T,H)$ defined by \eqref{d:def-T} and \eqref{d:def-H} is indeed a CFTP time with ambiguities. 

We now prove that, for small enough $\epsilon$ and $\kappa$, the pair $(T,H)$ satisfies $\g < 1$. Using the definition, then Lemma  \ref{l:estimation-2} , we have that 
$$\g = \sum_{\ell \geq 0} \E \left( \sum_{a \in \TT'_{\ell}} |A_{i_a}| \un(i_a \in \I^p) \right) \leq \kappa \sum_{\ell \geq 0} (1+\epsilon)^{\ell}   \P^u( |\X^u_{\infty}| > \ell  ).$$
From \eqref{e:queue-exp}, we see that $\g < 1$ for all $\epsilon < \mu$ and small enough $\kappa$. 

We now prove that $\E(\exp(\lambda T))<+\infty$ for all small enough $\epsilon$ and $\lambda$. 
We start with the observation that
$$\E( \exp(\lambda T)) \leq  \E \left(  \sum_{a \in \TT'} \exp(\lambda t_a)  \right) = \sum_{\ell \geq 0}    \E\left(  \sum_{a \in \TT'_{\ell}} \exp(\lambda t_a)  \right).$$
From Lemma \ref{l:estimation-3}, we deduce that, for all $\lambda < \r$,  
\begin{equation}\label{e:somme-tranches-2}\E( \exp(\lambda T)) \leq  \frac{  \r} {\r  - \lambda}  \sum_{\ell \geq 0} (1+\epsilon)^{\ell}   \E^u(\exp(\lambda t_{\ell-1}) \un(|\X^u_{\infty}| > \ell )).\end{equation}
By Schwarz's inequality, 
$$  \E^u(\exp(\lambda t_{\ell-1}) \un(|\X^u_{\infty}| > \ell )) \leq  \left(  \E^u(\exp( 2\lambda t_{\ell-1})  \right)^{1/2} \P^u(|\X^u_{\infty}| > \ell )^{1/2}.$$
Bounding above $t_{\ell-1}$ by the sum of $\ell$ independent exponential random variables with parameter $\r_u$ on one hand, and using \eqref{e:queue-exp} on the other hand, one obtains that, when $\lambda <  \r_u/2$, 
\begin{equation}\label{e:borne-sauvage}     \E^u(\exp(\lambda t_{\ell}) \un(|\X^u_{\infty}| > \ell )) \leq C^{1/2}  \left(\frac{\r_u}{\r_u-2 \lambda}\right)^{\ell}  \exp(-\mu(\ell+1)/2).     \end{equation} 
Combining \eqref{e:somme-tranches-2} and \eqref{e:borne-sauvage}, we have that $\E(\exp(\lambda T))<+\infty$ for all small enough $\epsilon$ and $\lambda$.

We now prove that $\Lambda_{H,time}(\lambda)<1$ for all small enough $\epsilon, \kappa, \lambda$. Using  Lemma \ref{l:estimation-4}, we obtain  that
\begin{equation}\label{e:somme-tranches-3}\Lambda_{H,time}(\lambda)   \leq    \frac{\kappa \r}{\r  - \lambda}  \sum_{\ell \geq 0} (1+\epsilon)^{\ell}   \E^u(\exp(\lambda t_{\ell}) \un(|\X^u_{\infty}| > \ell )).\end{equation}
Using again \eqref{e:borne-sauvage}, one concludes that $\Lambda_{H,time}(\lambda)<1$ for all small enough $\epsilon, \kappa, \lambda$.

Now let $R$ denote the depth of $\TT$, and  define $L := \beta_u(R)$ (remember that $\beta$ is defined in \eqref{e:borne-taille}). By definition of the exploration process with locking of ambiguities, one checks that  $L$ defines a stopping box and that  $H$ is measurable with respect to $\F^{-L, L}$.  Now, \eqref{e:premiere-substitution} and \eqref{e:seconde-substitution} show that $\left[\Phi_{T}^{0-}(\xi) \right](0)$ satisfies the required measurability properties.

Using the obvious inequality $R \leq |\TT|$, Lemma \ref{l:estimation-1} shows that 
$\Lambda_{L}(\lambda)$ is finite for small enough $\epsilon, \lambda$. Finally, Lemma \ref{l:estimation-6} shows that $\Lambda_{H,space}(\lambda,q)<1$ for all $q$, when $\epsilon, \kappa, \lambda$ are small enough.

\bibliographystyle{plain}
\bibliography{perturb-rev}

\end{document}